\documentclass{amsart}

\usepackage{amsmath,amssymb,amsfonts,hyperref}

\newtheorem{theorem}{Theorem}[section]
\newtheorem{lemma}[theorem]{Lemma}

\DeclareMathOperator{\bbC}{\mathbb{C}}
\DeclareMathOperator{\bbE}{\mathbb{E}}

\DeclareMathOperator{\bbZ}{\mathbb{Z}}
\DeclareMathOperator{\A}{\mathcal{A}}
\DeclareMathOperator{\Aut}{\operatorname{Aut}}
\DeclareMathOperator{\B}{\mathcal{B}}

\renewcommand{\H}{\mathcal{H}}

\DeclareMathOperator{\id}{\operatorname{id}}
\DeclareMathOperator{\J}{\mathcal{J}}
\DeclareMathOperator{\K}{\mathcal{K}}

\renewcommand{\O}{\mathcal{O}}
\DeclareMathOperator{\SOT}{\operatorname{SOT}}

\begin{document}
\title[PEP for Crossed Products]{The Pure Extension Property\\ for Discrete Crossed Products}
\author{Vrej Zarikian}
\address{Department of Mathematics, U. S. Naval Academy, Annapolis, MD 21402}
\email{zarikian@usna.edu}
\subjclass[2010]{Primary: 47L65, 46L55, 46L30 Secondary: 46L07}
\keywords{Pure state extension property, crossed product $C^*$-algebra, free action, discrete group}
\begin{abstract}
Let $G$ be a discrete group acting on a unital $C^*$-algebra $\A$ by $*$-automorphisms. In this note, we show that the inclusion $\A \subseteq \A \rtimes_r G$ has the pure extension property (so that every pure state on $\A$ extends uniquely to a pure state on $\A \rtimes_r G$) if and only if $G$ acts freely on $\mathcal{\widehat{A}}$, the spectrum of $\A$. The same characterization holds for the inclusion $\A \subseteq \A \rtimes G$. This generalizes what was already known for $\A$ abelian.
\end{abstract}

\maketitle

\section{Introduction}

Let $\A \subseteq \B$ be a \emph{$C^*$-inclusion}, i.e., an inclusion of unital $C^*$-algebras, with $1_{\A} = 1_{\B}$. We say that $\A \subseteq \B$ has the \emph{pure extension property} (PEP) if every pure state on $\A$ extends uniquely to a (pure) state on $\B$. The PEP was first considered by Kadison and Singer in \cite{KadisonSinger1959}, where they showed that $L^\infty[0,1] \subseteq B(L^2[0,1])$ fails the PEP, and asked whether or not $\ell^\infty \subseteq B(\ell^2)$ has the PEP. The latter question, famously unsolved for over 50 years, was settled affirmatively by Marcus, Spielman, and Srivastava in \cite{MarcusSpielmanSrivastava2015}.\\

Initially, the study of the PEP focused on the case of $\A$ abelian \cite{Anderson1979,ArchboldBunceGregson1982,Batty1982}. More recently, the general case has received attention \cite{BunceChu1998,Archbold1999}. Thanks to these efforts, we have various characterizations of the PEP, and know that it entails significant structural consequences for the inclusion. Very recently, several authors have advanced our understanding of the PEP for specific classes of inclusions \cite{Renault2008,AkemannWassermannWeaver2010,AkemannSherman2012,Popa2014,Pitts2017}.\\

This note continues the aforementioned line of inquiry, by characterizing (in terms of the dynamics) when the inclusion $\A \subseteq \A \rtimes_r G$ (resp. $\A \subseteq \A \rtimes G$) has the PEP. Here $G$ is a discrete group acting on a unital $C^*$-algebra $\A$ by $*$-automorphisms (abbreviated $G \curvearrowright \A$), and $\A \rtimes_r G$ (resp. $\A \rtimes G$) is the resulting reduced (resp. full) crossed product. When needed, $\alpha_g \in \Aut(\A)$ will denote the $*$-automorphism corresponding to $g \in G$.\\

If $\A$ is abelian, then $\A \cong C(X)$, the continuous complex-valued functions on a compact Hausdorff space $X$. In that case, the answer is already known. Indeed, an action $G \curvearrowright C(X)$ by $*$-automorphisms corresponds to an action $G \curvearrowright X$ by homeomorphisms. The inclusion $C(X) \subseteq C(X) \rtimes G$ has the PEP if and only if $G \curvearrowright X$ is free (meaning that $\{x \in X: g \cdot x = x\} = \emptyset$, for all $e \neq g \in G$) \cite[Cor. 6.2]{Batty1982} (see also \cite[Thm. 3.3.7]{Tomiyama1987} and \cite[Prop. 5.11]{Renault2008}).\\

In order to state our result, we remind the reader that the \emph{spectrum} of a unital $C^*$-algebra $\A$ is the set $\mathcal{\widehat{A}}$ of all unitary equivalence classes of non-zero irreducible representations of $\A$, equipped with the final topology induced by the surjection $PS(\A) \to \mathcal{\widehat{A}}:\phi \mapsto [\pi_\phi]$ arising from the GNS construction. An action $G \curvearrowright \A$ by $*$-automorphisms determines an action $G \curvearrowright \mathcal{\widehat{A}}$ by homeomorphisms. Namely $g \cdot [\pi] = [\pi \circ \alpha_g]$, for $g \in G$ and $\pi:\A \to B(\H)$ a non-zero irreducible representation.\\

Our result (Theorem \ref{PEP} below) says that $\A \subseteq \A \rtimes_r G$ (resp. $\A \subseteq \A \rtimes G$) has the PEP if and only if $G \curvearrowright \mathcal{\widehat{A}}$ is free (meaning that $\{[\pi] \in \mathcal{\widehat{A}}: [\pi \circ \alpha_g] = [\pi]\} = \emptyset$, for all $e \neq g \in G$). If $\A = C(X)$, then there is canonical homeomorphism $\mathcal{\widehat{A}} \cong X$, and so we have generalized the abelian case.\\

It should be noted that when moving from the abelian to the general case, many inequivalent notions of a ``free action'' present themselves, each with its own benefits and limitations \cite{Phillips2009}. Theorem \ref{PEP} singles out one of these notions as being harmonious with the PEP.\\

This paper can be regarded as a companion to our earlier paper \cite{Zarikian2017}, with which it shares many techniques. There we determine when the inclusion $\A \subseteq \A \rtimes_r G$ (resp. $\A \subseteq \A \rtimes G$) has a unique conditional expectation. This happens if and only if $G \curvearrowright \A$ is free\footnote{We say that an action $G \curvearrowright \A$ of a discrete group on a unital $C^*$-algebra by $*$-automorphisms is \emph{free} if $\alpha_g \in \Aut(\A)$ has no non-zero \emph{dependent elements}, for all $e \neq g \in G$. That is, if $b \in \A$ and $ba = \alpha_g(a)b$ for all $a \in \A$, then $b = 0$, unless $g = e$.} \cite[Thm. 3.1.2]{Zarikian2017}. Likewise, we determine when the inclusion $\A \subseteq \A \rtimes_r G$ (resp. $\A \subseteq \A \rtimes G$) has a unique \emph{pseudo-expectation}, in the sense of Pitts \cite{Pitts2017}. This happens if and only if $G \curvearrowright \A$ is properly outer\footnote{We say that $G \curvearrowright \A$ is \emph{properly outer} if for all $e \neq g \in G$, the only $\alpha_g$-invariant ideal $\J \subseteq \A$ such that $\alpha_g|_{\J}$ is \emph{quasi-inner} is $\J = \{0\}$.} \cite[Thm. 3.2.2]{Zarikian2017}. For an arbitrary $C^*$-inclusion $\A \subseteq \B$, we have that
\[
    \text{PEP} \implies \text{unique pseudo-expectation} \implies \text{at most one conditional expectation}.
\]
Thus, in passing, we re-prove the known implications
\[
    \text{$G \curvearrowright \mathcal{\widehat{A}}$ is free}
    \implies \text{$G \curvearrowright \A$ is properly outer}
    \implies \text{$G \curvearrowright\A$ is free}.
\]

\section{The Main Result}

In this section we prove our main result, Theorem \ref{PEP}. Before doing so, we will need a few preparatory lemmas. Our first preliminary result will make it slightly easier to determine when $G \curvearrowright \mathcal{\widehat{A}}$ is free.

\begin{lemma} \label{free}
Let $\A$ be a unital $C^*$-algebra and $\alpha \in \Aut(\A)$. Then
\[
    \{[\pi] \in \mathcal{\widehat{A}}: [\pi \circ \alpha] = [\pi]\} = \emptyset
\]
if and only if for every non-zero irreducible representation $\pi:\A \to B(\H)$ and every $T \in B(\H)$,
\[
    T\pi(a) = \pi(\alpha(a))T, ~ a \in \A \implies T = 0.
\]
\end{lemma}

\begin{proof}
($\Rightarrow$) Let $\pi:\A \to B(\H)$ be a non-zero irreducible representation, and suppose there exists $T \in B(\H)$ such that
\[
    T\pi(a) = \pi(\alpha(a))T, ~ a \in \A.
\]
Arguing as in \cite{ChodaKasaharaNakamoto1972}, we see that
\[
    T^*T = TT^* \in \pi(\A)' = \bbC I.
\]
Thus, either $T = 0$ or there exists a unitary $U \in B(\H)$ such that
\[
    U\pi(a) = \pi(\alpha(a))U, ~ a \in \A.
\]
Since $\{[\pi] \in \mathcal{\widehat{A}}: [\pi \circ \alpha] = [\pi]\} = \emptyset$, we conclude that $T = 0$.

($\Leftarrow$) Conversely, suppose there exists a non-zero irreducible representation $\pi:\A \to B(\H)$ such that $[\pi \circ \alpha] = [\pi]$. Then there exists a unitary $U \in B(\H)$ such that
\[
    \pi(\alpha(a)) = U\pi(a)U^*, ~ a \in \A.
\]
It follows that
\[
    U\pi(a) = \pi(\alpha(a))U, ~ a \in \A.
\]
\end{proof}

Our second preliminary result is a minor variation of \cite[Thm. 4.5]{Wittstock1981}, on decomposing a completely bounded (CB) bimodule map into completely positive (CP) bimodule maps. The main difference is that $\pi$ in Lemma \ref{decomposable} need not be faithful. The proof consists of an elementary reduction to the faithful case, so that \cite[Thm. 4.5]{Wittstock1981} can be invoked.

\begin{lemma} \label{decomposable}
Let $\A \subseteq \B$ be a $C^*$-inclusion, $\pi:\A \to B(\H)$ be a unital $*$-homomorphism, and $\theta:\B \to B(\H)$ be a CB map. Assume that $\theta$ is $\A$-bimodular with respect to $\pi$, meaning that for all $a \in \A$ and $x \in \B$, we have that
\[
    \theta(ax) = \pi(a)\theta(x) \text{ and } \theta(xa) = \theta(x)\pi(a).
\]
Then $\theta = (\theta_1 - \theta_2) + i(\theta_3 - \theta_4)$, where $\theta_j:\B \to B(\H)$ is a CP map which is $\A$-bimodular with respect to $\pi$, for all $1 \leq j \leq 4$.
\end{lemma}

\begin{proof}
We may assume that $\B \subseteq B(\K)$, for some Hilbert space $\K$. Define $\tilde{\pi}:\A \to B(\K) \oplus B(\H)$ by the formula
\[
    \tilde{\pi}(a) = a \oplus \pi(a), ~ a \in \A,
\]
and $\tilde{\theta}:\B \to B(\K) \oplus B(\H)$ by the formula
\[
    \tilde{\theta}(x) = x \oplus \theta(x), ~ x \in \B.
\]
Then $\tilde{\pi}$ is a faithful $*$-homomorphism, and $\tilde{\theta}$ is a CB map which is $\A$-bimodular with respect to $\tilde{\pi}$. Since $B(\K) \oplus B(\H)$ is injective, \cite[Thm. 4.5]{Wittstock1981} applies to show that $\tilde{\theta} = (\tilde{\theta}_1 - \tilde{\theta}_2) + i(\tilde{\theta}_3 - \tilde{\theta}_4)$, where $\tilde{\theta}_j:\B \to B(\K) \oplus B(\H)$ is a CP map which is $\A$-bimodular with respect to $\tilde{\pi}$, for all $1 \leq j \leq 4$. For each $1 \leq j \leq 4$, we have that $\tilde{\theta}_j = \alpha_j \oplus \theta_j$, where $\alpha_j:\B \to B(\K)$ and $\theta_j:\B \to B(\H)$ are CP maps. One easily verifies that $\theta_j$ is $\A$-bimodular with respect to $\pi$, for all $1 \leq j \leq 4$, and that $\theta = (\theta_1 - \theta_2) + i(\theta_3 - \theta_4)$.
\end{proof}

Our third and last preliminary result is a bimodule version of \cite[Lemma 5.1.6]{EffrosRuan2000}, on factoring a completely positive (CP) map as a unital completely positive (UCP) map followed by a conjugation. The proof is nearly identical.

\begin{lemma} \label{unitization}
Let $\A \subseteq \B$ be a $C^*$-inclusion, $\pi:\A \to B(\H)$ be a unital $*$-homomorphism, and $\theta:\B \to B(\H)$ be a CP map which is $\A$-bimodular with respect to $\pi$. Then there exists a UCP map $\tilde{\theta}:\B \to B(\H)$ which is $\A$-bimodular with respect to $\pi$, and such that
\[
    \theta(x) = \theta(1)^{1/2}\tilde{\theta}(x)\theta(1)^{1/2}, ~ x \in \B.
\]
In particular, $\tilde{\theta}|_{\A} = \pi$.
\end{lemma}

\begin{proof}
Since $\theta$ is $\A$-bimodular with respect to $\pi$, we have that $\theta(1) \in \pi(\A)'$. We claim that
\[
    \SOT-\lim_{n \to \infty} (\theta(1) + 1/n)^{-1/2}\theta(x)(\theta(1) + 1/n)^{-1/2}
\]
exists for all $x \in \B$. Indeed, if $0 \leq x \leq 1$, then $0 \leq \theta(x) \leq \theta(1)$, which implies $\theta(x)^{1/2} = t\theta(1)^{1/2} = \theta(1)^{1/2}t^*$ for some $t \in B(\H)$. Then
\[
    \SOT-\lim_{n \to \infty} (\theta(1) + 1/n)^{-1/2}\theta(x)^{1/2}
    = \SOT-\lim_{n \to \infty} (\theta(1) + 1/n)^{-1/2}\theta(1)^{1/2}t^* = pt^*,
\]
where $p \in \pi(\A)'$ is the support projection of $\theta(1)$. Likewise
\[
    \SOT-\lim_{n \to \infty} \theta(x)^{1/2}(\theta(1) + 1/n)^{-1/2}
    = \SOT-\lim_{n \to \infty} t\theta(1)^{1/2}(\theta(1) + 1/n)^{-1/2}
    = tp.
\]
It follows that
\[
    \SOT-\lim_{n \to \infty} (\theta(1) + 1/n)^{-1/2}\theta(x)(\theta(1) + 1/n)^{-1/2} = pt^*tp,
\]
which proves the claim. Now for $x \in \B$, define
\[
    \tilde{\theta}(x) = \SOT-\lim_{n \to \infty} (\theta(1) + 1/n)^{-1/2}\theta(x)(\theta(1) + 1/n)^{-1/2} + p^\perp\Phi(x)p^\perp,
\]
where $\Phi:\B \to B(\H)$ is any fixed UCP extension of $\pi$. Then $\tilde{\theta}$ is a UCP map which is $\A$-bimodular with respect to $\pi$, and such that
\[
    \theta(1)^{1/2}\tilde{\theta}(x)\theta(1)^{1/2} = \theta(x), ~ x \in \B.
\]
Indeed, if $0 \leq x \leq 1$, then (keeping the notation from above) we have that
\begin{eqnarray*}
    \theta(1)^{1/2}\tilde{\theta}(x)\theta(1)^{1/2}
    &=& \theta(1)^{1/2}(pt^*tp + p^\perp\Phi(x)p^\perp)\theta(1)^{1/2}\\
    &=& \theta(1)^{1/2}t^*t\theta(1)^{1/2} = \theta(x).
\end{eqnarray*}
\end{proof}

Finally we state and prove the main result.

\begin{theorem} \label{PEP}
Let $G$ be a discrete group acting on a unital $C^*$-algebra $\A$ by $*$-automorphisms. Then the following are equivalent:
\begin{enumerate}
\item[i.] $\A \subseteq \A \rtimes G$ has the PEP;
\item[ii.] $\A \subseteq \A \rtimes_r G$ has the PEP;
\item[iii.] $G \curvearrowright \mathcal{\widehat{A}}$ is free.
\end{enumerate}
\end{theorem}

\begin{proof}
(i $\implies$ ii) The PEP passes to quotients, by \cite[Lemma 3.1]{ArchboldBunceGregson1982}.

(ii $\implies$ iii) Suppose $\A \subseteq \A \rtimes_r G$ has the PEP. Fix $e \neq g \in G$, and assume that $\pi:\A \to B(\H)$ is a non-zero irreducible representation, such that $[\pi \circ \alpha_g] = [\pi]$. Then there exists a unitary $U \in B(\H)$ such that
\[
    \pi(\alpha_g(a)) = U\pi(a)U^*, ~ a \in \A.
\]
Define a CB map $\theta:\A \rtimes_r G \to B(\H)$ by the formula
\[
    \theta(x) = \pi(\bbE(xg^{-1}))U, ~ x \in \A \rtimes_r G,
\]
where $\bbE:\A \rtimes_r G \to \A$ is the canonical conditional expectation. Then for all $a \in \A$ and $x \in \A \rtimes_r G$, we have that
\[
    \theta(ax) = \pi(\bbE(axg^{-1}))U = \pi(a\bbE(xg^{-1}))U = \pi(a)\pi(\bbE(xg^{-1}))U = \pi(a)\theta(x).
\]
Likewise,
\begin{eqnarray*}
    \theta(xa)
    &=& \pi(\bbE(xag^{-1}))U = \pi(\bbE(xg^{-1}gag^{-1}))U\\
    &=& \pi(\bbE(xg^{-1}\alpha_g(a)))U = \pi(\bbE(xg^{-1})\alpha_g(a))U\\
    &=& \pi(\bbE(xg^{-1}))\pi(\alpha_g(a))U = \pi(\bbE(xg^{-1}))U\pi(a)\\
    &=& \theta(x)\pi(a).
\end{eqnarray*}
That is, $\theta$ is $\A$-bimodular with respect to $\pi$. Furthermore, $\theta(g) = U \neq 0$. By Lemma \ref{decomposable}, $\theta = (\theta_1 - \theta_2) + i(\theta_3 - \theta_4)$, where $\theta_j:\A \rtimes_r G \to B(\H)$ is a CP map which is $\A$-bimodular with respect to $\pi$, for $1 \leq j \leq 4$. Without loss of generality, $\theta_1(g) \neq 0$. By Lemma \ref{unitization}, there exists a UCP map $\tilde{\theta}_1:\A \rtimes_r G \to B(\H)$ which is $\A$-bimodular with respect to $\pi$, and such that $\theta_1(x) = \theta_1(1)^{1/2}\tilde{\theta}_1(x)\theta_1(1)^{1/2}$, $x \in \A \rtimes_r G$. Obviously $\tilde{\theta}_1(g) \neq 0$. Let $\xi \in \H$ be a unit vector such that $\langle \tilde{\theta}_1(g)\xi, \xi \rangle \neq 0$. Then $a \mapsto \langle \pi(a)\xi, \xi \rangle$ is a pure state on $\A$ with distinct state extensions $x \mapsto \langle \pi(\bbE(x))\xi, \xi \rangle$ and $x \mapsto \langle \tilde{\theta}_1(x)\xi, \xi \rangle$ on $\A \rtimes_r G$, a contradiction.

(iii $\implies$ i) Suppose $G \curvearrowright \mathcal{\widehat{A}}$ is free. Let $\phi \in PS(\A)$. Then the corresponding GNS representation $\pi_\phi:\A \to B(\H_\phi)$ is a non-zero irreducible representation. Let $\varPhi \in S(\A \rtimes G)$ be an extension of $\phi$, and let $\pi_\varPhi:\A \rtimes G \to B(\H_\varPhi)$ be the resulting GNS representation. For all $a \in \A$, we have that
\[
    \pi_\varPhi(a)V = V\pi_\phi(a),
\]
where $V:\H_\phi \to \H_\varPhi$ is the unique isometry such that
\[
    V\pi_\phi(a)\xi_\phi = \pi_\varPhi(a)\xi_\varPhi, ~ a \in \A.
\]
Taking adjoints,
\[
    V^*\pi_\varPhi(a) = \pi_\phi(a)V^*, ~ a \in \A.
\]
Now fix $e \neq g \in G$. For all $a \in \A$, we have that
\begin{eqnarray*}
    gag^{-1} = \alpha_g(a)
    &\implies& ga = \alpha_g(a)g\\
    &\implies& \pi_\varPhi(g)\pi_\varPhi(a) = \pi_\varPhi(\alpha_g(a))\pi_\varPhi(g)\\
    &\implies& V^*\pi_\varPhi(g)\pi_\varPhi(a)V = V^*\pi_\varPhi(\alpha_g(a))\pi_\varPhi(g)V\\
    &\implies& V^*\pi_\varPhi(g)V\pi_\phi(a) = \pi_\phi(\alpha_g(a))V^*\pi_\varPhi(g)V.
\end{eqnarray*}
Since $\{[\pi] \in \mathcal{\widehat{A}}: [\pi \circ \alpha_g] = [\pi]\} = \emptyset$, we have that $V^*\pi_\varPhi(g)V = 0$, by Lemma \ref{free}. Thus for any $a \in \A$,
\[
    V^*\pi_\varPhi(ag)V = V^*\pi_\varPhi(a)\pi_\varPhi(g)V = \pi_\phi(a)V^*\pi_\varPhi(g)V = 0.
\]
Therefore
\[
    \varPhi(ag) = \langle \pi_\varPhi(ag)\xi_\varPhi, \xi_\varPhi \rangle
    = \langle \pi_\varPhi(ag)V\xi_\phi, V\xi_\phi \rangle
    = \langle V^*\pi_\varPhi(ag)V\xi_\phi, \xi_\phi \rangle = 0.
\]
Since $a \in \A$ and $e \neq g \in G$ were arbitrary, $\varPhi = \phi \circ \tilde{\bbE}$, where $\tilde{\bbE}:\A \rtimes G \to \A$ is the canonical conditional expectation.
\end{proof}

\section{An Example}

In this final section, we use Theorem \ref{PEP} to analyze the inclusion $\O_2 \subseteq \O_2 \rtimes \bbZ_2$, showing that it fails the PEP, although it has many of the features of a PEP inclusion. Here $\O_2$ is the Cuntz algebra, i.e., the universal $C^*$-algebra generated by isometries $s_0, s_1$ satisfying the relation $s_0s_0^* + s_1s_1^* = 1$, and the action of $\bbZ_2 = \{0,1\}$ on $\O_2$ switches the generators. More precisely, $\alpha_0 = \id$ and $\alpha_1$ is the unique $*$-automorphism of $\O_2$ such that $\alpha_1(s_0) = s_1$. This action was studied in detail by Choi and Latr\'{e}moli\`{e}re in \cite{ChoiLatremoliere2012}, and we benefit tremendously from their insights. In particular, they show that $\O_2 \rtimes \bbZ_2 \cong \O_2$ \cite[Thm. 2.1]{ChoiLatremoliere2012}. Thus $\O_2 \rtimes \bbZ_2$ is simple, and we need not distinguish between $\O_2 \rtimes \bbZ_2$ and $\O_2 \rtimes_r \bbZ_2$.\\

By \cite[Prop. 2.2]{ChoiLatremoliere2012}, $\bbZ_2 \curvearrowright \O_2$ is outer (i.e., not inner). In fact, since $\O_2$ is simple, $\bbZ_2 \curvearrowright \O_2$ is properly outer, and therefore free (see \cite[Rmk. 4.1.3]{Zarikian2017}). Also, by \cite[Ex. 5.7]{Izumi2004}, $\bbZ_2 \curvearrowright \O_2$ has the \emph{Rokhlin property}. On the other hand, $\bbZ_2 \curvearrowright \mathcal{\widehat{O}}_2$ is not free. Indeed, as shown in \cite[App. A]{ChoiLatremoliere2012}, there is a non-zero irreducible representation $\pi:\O_2 \to B(L^2[-1,1])$ such that $[\pi \circ \alpha_1] = [\pi]$. Namely, for $f \in L^2[-1,1]$ and $t \in [-1,1]$,
\[
    (\pi(s_0)f)(t) = \sqrt{2}f(2t-1)\chi_{[0,1]}(t)
\]
and
\[
    (\pi(s_1)f)(t) = \sqrt{2}f(2t+1)\chi_{[-1,0]}(t).
\]
Then
\[
    \pi(\alpha_1(x)) = U\pi(x)U^*, ~ x \in \O_2,
\]
where $U \in B(L^2[-1,1])$ is the unitary operator defined by
\[
    (Uf)(t) = f(-t), ~ f \in L^2[-1,1], ~ t \in [-1,1].
\]
(It would be interesting to know if $\bbZ_2 \curvearrowright \mathcal{\widehat{O}}_2$ is \emph{topologically free}, in the sense of \cite{ArchboldSpielberg1993}.)\\

In light of the previous paragraph, we draw the following conclusions:
\begin{itemize}
\item $\O_2 \subseteq \O_2 \rtimes \bbZ_2$ admits a unique conditional expectation \cite[Thm. 3.1.2]{Zarikian2017};
\item $\O_2 \subseteq \O_2 \rtimes \bbZ_2$ admits a unique pseudo-expectation, in the sense of Pitts \cite[Thm. 3.2.2]{Zarikian2017};
\item $\O_2 \subseteq \O_2 \rtimes \bbZ_2$ fails the PEP, by Theorem \ref{PEP}.
\end{itemize}

\end{document}